\documentclass[10pt,a4paper]{article}

\usepackage[left=2.45cm, top=2.45cm,bottom=2.45cm,right=2.45cm]{geometry}

\usepackage{mathtools,amssymb,amsthm,mathrsfs}
\usepackage[british]{babel}
\usepackage{amsfonts}              
\numberwithin{equation}{section}
\numberwithin{figure}{section}

\newtheorem {theorem}{Theorem}[section]
\newtheorem {proposition}[theorem]{Proposition}
\newtheorem {lemma}[theorem]{Lemma}

\newcommand{\nc}{\newcommand}
\nc{\grad}{{\mbox{grad}\,}}
\nc{\R}{{\mathbb R}}
\nc{\Rn}{{\mathbb R}^n}
\nc{\N}{{\mathbb N}}
\nc{\Z}{{\mathbb Z}}
\nc{\K}{{\cal K}}
\nc{\kpo}{(\K^d_{2,gp})_o}
\nc{\krep}{\K^d_{3,gp}}
\nc{\BP}{\mathbb{P}}
\nc{\BQ}{\mathbb{Q}}
\nc{\BE}{\mathbb{E}}
\nc{\cH}{\mathcal{H}}
\nc{\cL}{\mathcal{L}}
\nc{\cE}{\mathcal{E}}
\nc{\BS}{\mathbb{S}}
\nc{\LL}{L^\circ}
\nc{\bB}{B}
\nc{\sph}{\mathbb{S}^{n-1}}
\nc{\Q}{\mathbb{Q}}
\nc{\cK}{\mathcal{K}}
\nc{\vaps}{\varepsilon}
\nc{\cV}{\mathcal{V}}
\nc{\cP}{\mathcal{P}}
\nc{\cR}{\mathcal{R}}
\nc{\cS}{\mathcal{S}}
\nc{\cB}{\mathcal{B}}
\nc{\fB}{\mathfrak{B}}
\nc{\cC}{\mathcal{C}}
\nc{\ver}{{\rm{vert\, }}}
\nc{\iF}{\,\raisebox{0.0pt}{$\square$}\,}
\nc{\ind}{\mathbf{1}}	
\nc{\fed}{\llcorner}

\DeclareMathOperator{\diam}{diam}

\begin{document}




\title{\bfseries A reverse Minkowski-type inequality}

\author{Daniel Hug\footnotemark[1]\; and K\'aroly B\"or\"oczky \footnotemark[2]}

\renewcommand{\thefootnote}{\fnsymbol{footnote}}
\footnotetext[1]{Karlsruhe Institute of Technology (KIT), Department of Mathematics, D-76128 Karlsruhe, Germany. E-mail: daniel.hug@kit.edu}

\footnotetext[2]{ Alfr\'ed R\'enyi Institute of Mathematics, Hungarian Academy of Sciences, Reltanoda u.~13-15, H-1053 Budapest, Hungary, and Department of Mathematics, Central European University, Nador u 9, H-1051, Budapest, Hungary.
E-mail: boroczky.karoly.j@renyi.mta.hu, Supported in part by
NKFIH grants 116451, 121649, 132002 and 129630.}

\maketitle

\begin{abstract}
The famous Minkowski inequality provides a sharp lower bound for the mixed volume $V(K,M[n-1])$ of two
convex bodies $K,M\subset\R^n$  in terms of powers of the volumes of the individual
bodies $K$ and $M$. The special case where $K$ is the unit ball yields the isoperimetric inequality.
In the plane, Betke and Weil (1991) found a sharp upper bound for the mixed area of $K$ and $M$ in terms
of the perimeters of $K$ and $M$. We extend this result to general dimensions by proving a sharp upper
bound for the mixed volume $V(K,M[n-1])$ in terms of the mean width of $K$ and the surface area of $M$.
The equality case is completely characterized. In addition, we establish a stability improvement of this
and related geometric inequalities of isoperimetric type.
\footnotetext[3]{{\em Key words and phrases.} {Geometric inequality, Brunn-Minkowski theory, Minkowski inequality, mean width, surface area, mixed volume, stability result.}}
\footnotetext[4]{{\em 2010 Mathematics Subject Classification.} Primary:  52A20, 52A38, 52A39, 52A40; secondary: 60D05, 52A22.}
\end{abstract}

\maketitle


\bibliographystyle{amsplain}





\section{Introduction}

Mixed volumes of convex bodies in Euclidean space $\R^n$ are fundamental functionals which encode geometric
information about the involved convex bodies in a non-trivial way.
Let $\cK^n$ denote the space of nonempty compact convex subsets of $\R^n$. For $K,M\in\cK^n$ and $\alpha,\beta\ge 0$, the volume $V(\alpha K+\beta M)$  of the Minkowski sum $\alpha K+\beta M$ has the polynomial expansion
\begin{equation}\label{expansionpol}
V(\alpha K+\beta M)=\sum_{i=0}^n\binom{n}{ i}V(K[i],M[n-i])\alpha^{i}\beta^{n-i},
\end{equation}
by which the special mixed volumes $V(K[i],M[n-i])$ involving $i$ copies  of $K$ and $n-i$ copies  of $M$, $i\in\{0,\ldots,n\}$,  are uniquely determined.  We refer to \cite{S14} for a general introduction to mixed volumes and a thorough study of
their properties. In the following, we simply write $V(K,M[n-1])$ if $K$ appears with multiplicity one. In particular,  \eqref{expansionpol} implies that
$$
nV(M,K[n-1])=\lim_{\varepsilon\to 0+}\frac{1}{\varepsilon}(V(K+\varepsilon M)-V(K)).
$$
This explains why $nV(B^n,K[n-1])$ yields the surface area $F(K)$ of $K$ if $M=B^n$ is the Euclidean unit ball. More generally, the special
choice $M=B^n$ in \eqref{expansionpol}  leads to the intrinsic volumes
$$
V_i(K)= \frac{1}{\kappa_{n-i}}\binom{n}{i}  V(K[i],B^n[n-i]),\qquad i\in\{0,\ldots,n\},
$$
where $\kappa_m$ is the volume of $B^m$ in $\R^m$. We note that $V_n=V$ is the volume functional, and $V_1(K)=(n/\kappa_{n-1})V(K,B^n[n-1])=n\kappa_n/(2\kappa_{n-1})w(K)$, where $w(K)$ is the mean width of $K$.
Furthermore, we have $2V_{n-1}(K)=nV(B^n,K[n-1])= F(K)$, and $V_{i}(K)=\cH^{i}(K)$ if $K$ is $i$-dimensional. Here,   $\cH^i$ is
 the $i$-dimensional Hausdorff-measure, which is normalized in such a way that it coincides with the Lebesgue measure on $\R^i$.

One of the fundamental results for mixed volumes is Minkowski's inequality
\begin{equation}
\label{Minkowski-ineq}
V(K,M[n-1])^n\geq V(K)V(M)^{n-1} \mbox{ \ for $K,M\in\cK^n$}.
\end{equation}
If ${\rm int}\,K,\,{\rm int}\, M\neq\emptyset$ (that is, $K,M$ have nonempty interiors), then equality holds if and only if $K$ and $M$ are homothetic  (we refer to \cite{S14} for notions and results  in the Brunn-Minkowski theory which are used in the following without further explanation).

As a planar and reverse counterpart of the Minkowski inequality \eqref{Minkowski-ineq},
Betke and Weil proved the following theorem (see \cite[Theorem 1]{BW}) which yields a sharp
upper bound for the mixed area of $K, M \in\K^2$ in terms of the perimeters $F(K),F(M)$ of $K$ and $M$.

\begin{theorem}[Betke, Weil (1991)]
\label{Betke-Weil}
If $K, M \in\K^2$, then
$$
  V(K, M) \leq \frac 18\, F(K)\, F(M)
$$
with equality if and only if $K$ and $M$ are orthogonal (possibly degenerate) segments.
\end{theorem}

Note that for $K, M \in\K^2$ we have $F(K)\, F(M)=4\,V_1(K)\, V_1(M)$. We extend Theorem \ref{Betke-Weil} to general dimensions and thus obtain the following reverse Minkowski-type inequality.

\begin{theorem}\label{mvolupper}
If $K,M\in\cK^n$, then
$$
V(K,M[n-1])\le \frac{1}{n}\,V_1(K)\,V_{n-1}(M);
$$
if $\text{\rm dim}(K)\ge 1$ and $\text{\rm dim}(M)\ge n-1$, then equality holds  if and only if $K$ is a segment and $M$ is contained in a hyperplane orthogonal to $K$.
\end{theorem}

A crucial ingredient for the proof of Theorem \ref{mvolupper} is a lower bound for
$V_1(K)$ in terms of the circumradius of $K$, which was first proved by Linhart \cite{Linhart} (see also  Theorem \ref{thmbkr0} (i) below).

For Minkowski's inequality \eqref{Minkowski-ineq}  various stability versions  have been found, the first is due to Minkowski himself. Here we cite only two such results.  Groemer \cite{Groemer1990} proved that
if $K,M\in\cK^n$ with ${\rm int}\,K,\,{\rm int}\, M\neq\emptyset$ and $\varepsilon>0$ is sufficiently small, then
\begin{equation}
\label{Minkowskistabcond}
V(K,M[n-1])^n\leq (1+\varepsilon)V(K)V(M)^{n-1}
\end{equation}
implies that there exist $y,z\in \R^n$ and $\lambda>0$ such that
$$
 \lambda (K-z)\subset M-y\subset \left(1+\gamma\varepsilon^{\frac1{n+1}}\right)\lambda (K-z)
$$
where $\gamma>0$ depends only on $n$.

In addition, Figalli, Maggi, Pratelli \cite{FMP2010} showed that  \eqref{Minkowskistabcond} implies that
there is some $x\in \R^n$ such that
$$
\cH^n\left(M\Delta(x+\lambda K)\right)\leq \gamma \,\sqrt{\varepsilon}\, V(M)
$$
 where  $\lambda=(V(M)/V(K))^{1/n}$, $\Delta$ stands for the symmetric difference and $\gamma>0$ depends only on $n$.

 These stability results improve Minkowski's first inequality provided some information about the deviation of the
 shapes of $K$ and $M$ (up to homothety) is available and at the same time they provide additional information on how
 close  $K$ and $M$ are if almost equality holds in Minkowski's inequality.

We obtain the following stability version of the reverse Minkowski inequality given in Theorem~\ref{mvolupper}. Here and in the following, we write $R(K)$ for the circumradius of $K\in\cK^n$. Further, we denote by $h_K$ the support function of $K$ (see Section \ref{secpreliminaries}), hence $w_K(e)=h_K(e)+
h_K(-e)$ is the width of $K$ in direction $e\in\BS^{n-1}$, and we write $K|e^\perp$ for the orthogonal projection of $K$ to $e^\perp$. If $S$ is a nondegenerate segment parallel to $e\in\BS^{n-1}$, then we denote by $B_S^{n-1}$ the unit ball in the $(n-1)$-dimensional linear subspace orthogonal to $e$.

\begin{theorem}
\label{mvolupperstab}
Let $K,M\in\cK^n$ with $\dim(K)\ge 1$, $\dim(M)\ge n-1$. Suppose that
$$
V(K,M[n-1])\ge  (1-\varepsilon)\,\frac{1}{n}\,V_1(K)V_{n-1}(M)
$$
for some $\varepsilon\in [0,\varepsilon_0]$. Then there exist unit vectors
$e,v\in\BS^{n-1}$ and a segment $S$ of length $(2-\gamma_1\varepsilon^2)R(K)$ parallel to $e$
such that $w_M(v)\le \gamma_2\,r\,\varepsilon^{\frac{1}{4}}$,  $\langle e,v\rangle \ge 1-\gamma_2\sqrt{\varepsilon}$  and
$S\subset K\subset S+\gamma_2R(K)\sqrt{\varepsilon}B^{n-1}_S$,
 where $r$ is the maximal radius of an $(n-1)$-ball in $M\vert e^\perp$, and $\varepsilon_0,\gamma_1,\gamma_2>0$ are constants
depending on $n$.
\end{theorem}

Note that under the assumptions of the theorem, $M$ is contained in a slab of width at most $ \gamma_2r\varepsilon^{\frac{1}{4}}$
and  this slab is almost orthogonal (in a quantitative sense) to the segment $S$. Furthermore, constants such as $\varepsilon_0$ or $\gamma_1,\gamma_2$ are used generically and are not necessarily the same when they are used in different statements.

A crucial ingredient in the proofs of Theorem~\ref{mvolupper} and Theorem~\ref{mvolupperstab} is the following result,
which is interesting in its own right.

\begin{theorem}\label{thmbkr0}
Let $K\in\cK^n$ with $\dim(K)\ge 1$.
\begin{enumerate}
\item[{\rm (i)}] Then $V_1(K)\ge 2 R(K)$,
with equality if and only if $K$ is a segment.
\item[{\rm (ii)}] If  $V_1(K)\le (2+\varepsilon) R(K)$ for some $\varepsilon\in[0,\varepsilon_0]$, then there exists a segment
$S$ of length
$(2-\gamma_1\varepsilon^2)R(K)$ such that $S\subset K\subset S+\gamma_2R(K)\sqrt{\varepsilon}B^{n-1}_S$,
where $\varepsilon_0,\gamma_1,\gamma_2>0$ are constants depending on $n$.
\end{enumerate}
\end{theorem}

The inequality between the circumradius and  the first intrinsic volume (or the mean width) of a convex body, which is stated
in Theorem \ref{thmbkr0} (i), is due to J.~Linhart \cite{Linhart}.
Our proof for part (i) follows Linhart's idea, but we introduce several modifications so as to simplify the discussion of the equality case  and prepare for the proof of part (ii).  The proof of the assertion in part (ii) provides a substantial strengthening and refinement of Linhart's argument.

The order of the error bound  in Theorem~\ref{thmbkr0} (ii) is $\sqrt{\varepsilon}$. This is the optimal order, as can be seen by considering  isosceles triangles.

Geometric stability results have recently found applications in stochastic geometry, in particular in the study of
 shapes of large cells in  random tessellations \cite{HS2007}. We refer to the arxiv version of the present paper for an application of Theorem \ref{thmbkr0}  to a probabilistic deviation result
for  Poisson hyperplane tessellations in $\R^n$.

We note that Betke and Weil \cite{BW} also proved that
if $K \in\cK^2$, then
\begin{equation}
\label{Betke-Weil2}
  V(K, -K) \leq \frac {\sqrt 3}{18}\, F^2(K),
\end{equation}
and under the additional assumption that $K$ is a two-dimensional polygon they showed that
equality holds in \eqref{Betke-Weil2} if and only if
$K$ is an equilateral triangle.

Betke and Weil \cite{BW} suggested as a problem to characterize the equality cases of
\eqref{Betke-Weil2} among all planar compact convex sets $K\in\K^2$.
This goal is achieved in the forthcoming manuscript \cite{BBH}.

The paper is structured as follows. Some basic notions which are used in the following are introduced in Section~\ref{secpreliminaries}. Then Theorem~\ref{thmbkr0} is proved in Section~\ref{secRV1}.
Our main results, Theorem~\ref{mvolupper} and its stability version Theorem~\ref{mvolupperstab},
are established in Section~\ref{secmain}.

\section{Preliminaries}
\label{secpreliminaries}

For basic notions and results from the Brunn-Minkowski theory, we refer to the monograph  \cite{S14}. We work in Euclidean space $\R^n$, $n\ge 2$, with
scalar product $\langle\cdot\,,\cdot\rangle$ and induced Euclidean norm $\|\cdot\|$. The unit ball is denoted by $B^n$, its boundary is the unit sphere $\BS^{n-1}=\partial B^n$. For a set $A$ in a topological space we denote its closure by $\text{cl}
(A)$.
 If $u\in\BS^{n-1}$, then $u^\bot$ denotes the  $(n-1)$-dimensional linear subspace orthogonal to $u$, and we write $X|u^\bot$ for the orthogonal projection of $X\subset\R^n$ into $u^\bot$.
 The support function of a convex body $K \in\cK^n$   is
$h_K(x)=h(K,x)=\max\{\langle x,y\rangle:y\in K\}$ for $x\in\R^n$. The convex hull of $x_1,\ldots,x_k\in\R^n$ will be denoted
by $[x_1,\ldots,x_k]$, hence $h([x_1,\ldots,x_k],x)=\max\{\langle x_i,x\rangle:i=1,\ldots,k\}$. We write $\text{lin}\{x_1,\ldots,x_k\}$ for the linear hull of $x_1,\ldots,x_k\in\R^n$. For unit vectors $u,v\in\BS^{n-1}$ we denote by $\angle(u,v)\in [0,\pi]$ the angle enclosed by $u,v$ so that $\cos\angle(u,v)=\langle u,v\rangle$.

The surface area measure $S_{n-1}(K,\cdot)$ of $K\in\cK^n$ is the (unique) finite Borel measure
on $\BS^{n-1}$ such that, for all  $M \in\cK^n$,
$$
V(M,K[n-1])=\frac1n\int_{\BS^{n-1}}h_M(u)\, S_{n-1}(K,du).
$$
The map $K\mapsto S_{n-1}(K,\cdot)$ from $\cK^n$ to the space of finite Borel measures on $\BS^{n-1}$ is weakly continuous.
We note that $S_{n-1}(K,\BS^{n-1})=F(K)$ is the surface area of $K$, and if $K\in\cK^n$ and $e\in \BS^{n-1}$, then
\begin{equation}
\label{projection}
2\,\mathcal{H}^{n-1}(K\vert e^\bot)=\int_{\BS^{n-1}}|\langle e,u\rangle|\, S_{n-1}(K,du).
\end{equation}
We provide some additional information about the surface area measure of a convex body $K\in\cK^n$.
If ${\rm dim}\,K\leq n-2$, then $S_{n-1}(K,\cdot)\equiv 0$.
If ${\rm dim}\,K= n-1$ and the affine hull of $K$ is parallel to $u^\bot$, for   $u\in\BS^{n-1}$, then
$S_{n-1}(K,\cdot)$ is the even measure concentrated on $\{-u,u\}$ with
$S_{n-1}(K,\{\pm u\})=\cH^{n-1}(\{u\})$. Now suppose that
${\rm int}\,K\neq\emptyset$. Then for each $x\in\partial K$, there exists an
exterior unit normal $u\in\BS^{n-1}$ such that $h_K(u)=\langle x,u\rangle$. Moreover, for
$\cH^{n-1}$-almost all $x\in\partial K$ the exterior unit normal of $K$ at $x$ is uniquely determined. In this case,
$x$ is called a regular boundary point and the exterior unit normal of $K$ at $x$ is denoted by
$\nu_K(x)$. We write
$\partial' K$ to denote the set of regular boundary points of $K$.
Moreover, if $g:\,\BS^{n-1}\to\R$ is a bounded Borel function, then
$$
\int_{\BS^{n-1}}g(u)\, S_{n-1}(K,du)=
\int_{\partial' K}g(\nu_K(x))\,\cH^{n-1}(dx).
$$
Since $2V_{n-1}(K)=S_{n-1}(K,\BS^{n-1})$, we deduce from \eqref{projection}  that if $e\in \BS^{n-1}$, then
\begin{equation}
\label{projectionVn-1}
\mathcal{H}^{n-1}(K\vert e^\bot)\leq V_{n-1}(K),
\end{equation}
with equality for ${\rm dim}\,K\ge n-1$ if and only if
${\rm dim}\,K= n-1$ and
$e$ is normal to $K$.
In addition, for $K\in\cK^n$ and $e\in \BS^{n-1}$,  we have
$$
\cH^{n-1}(K\vert e^\bot)=\int_{K}|\langle e,u\rangle|\, \mathcal{H}^{n-1}(dx)
$$
if ${\rm dim}\,K= n-1$ and $u\in\BS^{n-1}$ is normal to $K$, and
$$
\cH^{n-1}(K\vert e^\bot)=\frac12\int_{\partial' K}|\langle e,\nu_{K}(x)\rangle|\, \mathcal{H}^{n-1}(dx)
$$
if ${\rm dim}\,K= n$.

\section{Proof of Theorem~\ref{thmbkr0}}
\label{secRV1}

For $z\in \BS^{n-1}$ and $\alpha\in(0,\pi)$, let $B(z,\alpha)=\{x\in \BS^{n-1}:\langle x,z\rangle\geq\cos\alpha\}$ be the spherical cap (geodesic ball) centered at $z$  of radius $\alpha$.
For a spherical set $\Pi\subset \BS^{n-1}$, we write ${\rm int}_{s}\Pi$ for the interior
and $\partial _{s}\Pi$ for the boundary of $\Pi$, with respect to $\BS^{n-1}$ and its topology induced by the geodesic metric. For a point $x\in \BS^{n-1}$, the point $-x$ is the antipode of $x$. We call $\Pi\subset \BS^{n-1}$  starshaped with respect to  $x_0\in \BS^{n-1}$ if $x_0\in \Pi$,
$-x_0\not\in \Pi$, and  for any $x\in \Pi\setminus \{ x_0\}$ the spherical geodesic arc connecting $x$ and $x_0$ is contained in $\Pi$.

If not specified otherwise, constants $c_1,c_2,\ldots $ used below will only depend on the dimension.

The following observation is a key step in proving Theorem~\ref{thmbkr0} (i).

\begin{lemma}
\label{monotonicity}
If $\alpha\in(0,\frac{\pi}2]$, $n\geq 2$, $z\in \BS^{n-1}$ and $\Pi\subset B(z,\alpha)$ is compact and starshaped with respect to $z$, then
\begin{equation}\label{keyin}
 \displaystyle{\int_{\Pi}\langle z,u\rangle\,\mathcal{H}^{n-1}(du)} \geq
\frac{\displaystyle{\int_{B(z,\alpha)}\langle z,u\rangle\,\mathcal{H}^{n-1}(du)}}{\mathcal{H}^{n-1}(B(z,\alpha))}\, {\mathcal{H}^{n-1}(\Pi)}.
\end{equation}
If $z\in {\rm int}_{s}\Pi$, then equality holds if and only if $\Pi=B(z,\alpha)$.
\end{lemma}

\begin{proof} For the proof, we can assume that $\mathcal{H}^{n-1}(\Pi)>0$.
For $u\in z^\bot\cap \BS^{n-1}$,  the ``spherical radial function" of $\Pi$ evaluated at $u$  is
defined by
$$
\varphi(u)=\max\{t\in [0,\alpha]:\,z\, \cos t+u\, \sin t\,\in\Pi\}\in[0,\alpha].
$$
In addition, let
$$
\Xi_0=\{u\in z^\bot\cap \BS^{n-1}:\,\varphi(u)>0\}.
$$
We set $\varrho(s)=(\sin s)^{n-2}$ for  $s\in(0,\pi)$.
Then an application of the transformation formula shows that
\begin{align*}
\int_{\Pi}\langle z,u\rangle\,\mathcal{H}^{n-1}(du)&=\int_{\Xi_0}
\int_{0}^{\varphi(u)}(\cos s) \varrho(s)\,ds\,\mathcal{H}^{n-2}(du),\\
\mathcal{H}^{n-1}(\Pi)&=\int_{\Xi_0}
\int_{0}^{\varphi(u)}\varrho(s)\,ds\,\mathcal{H}^{n-2}(du).
\end{align*}
Since $\cos s$ is decreasing in $s$, for  $u\in \Xi_0$ with $\varphi(u)<\alpha$ we have
$$
\frac{\int_{\varphi(u)}^{\alpha}(\cos s) \varrho(s)\,ds}
{\int_{\varphi(u)}^{\alpha}\varrho(s)\,ds}
< \cos\varphi(u) <
\frac{\int_{0}^{\varphi(u)}(\cos s) \varrho(s)\,ds}
{\int_{0}^{\varphi(u)}\varrho(s)\,ds},
$$
which in turn yields that
\begin{align*}
&\frac{\int_{0}^{\alpha}(\cos s) \varrho(s)\,ds}
{\int_{0}^{\alpha}\varrho(s)\,ds}=
\frac{\int_{0}^{\varphi(u)}\varrho(s)\,ds}{\int_{0}^{\alpha}\varrho(s)\,ds}
\,
\frac{\int_{0}^{\varphi(u)}(\cos s) \varrho(s)\,ds}
{\int_{0}^{\varphi(u)}\varrho(s)\,ds}\\
&\qquad\qquad\qquad\qquad\qquad+
\frac{\int_{\varphi(u)}^{\alpha}\varrho(s)\,ds}{\int_{0}^{\alpha}\varrho(s)\,ds}
\,
\frac{\int_{\varphi(u)}^{\alpha}(\cos s) \varrho(s)\,ds}
{\int_{\varphi(u)}^{\alpha}\varrho(s)\,ds}
\leq \frac{\int_{0}^{\varphi(u)}(\cos s) \varrho(s)\,ds}
{\int_{0}^{\varphi(u)}\varrho(s)\,ds}.
\end{align*}
The inequality in fact holds for all $u\in \Xi_0$, and it is strict if $\varphi(u)<\alpha$.
Therefore
\begin{align*}
\int_{\Pi}\langle z,u\rangle\,\mathcal{H}^{n-1}(du)&\geq
\frac{\int_{0}^{\alpha}(\cos s) \varrho(s)\,ds}
{\int_{0}^{\alpha}\varrho(s)\,ds}
\int_{\Xi_0}
\int_{0}^{\varphi(u)}\varrho(s)\,ds\,\mathcal{H}^{n-2}(du)\\
&=
\frac{\int_{B(z,\alpha)}\langle z,u\rangle\,\mathcal{H}^{n-1}(du)}{\mathcal{H}^{n-1}(B(z,\alpha))}
\, \mathcal{H}^{n-1}(\Pi),
\end{align*}
which proves \eqref{keyin}. If  $z\in {\rm int}_{s}\Pi$, then $\Xi_0=z^\perp\cap \BS^{n-1}$. Hence, if \eqref{keyin} holds with equality,
then $\varphi(u)=\alpha$ for $\cH^{n-2}$-almost all $u\in z^\perp\cap \BS^{n-1}$. But since $\Pi\subset B(z,\alpha)$ is compact,
it follows that $\Pi= B(z,\alpha)$.
\end{proof}

We note that for any $z\in \BS^{n-1}$, we have
\begin{equation}
\label{intratio}
\frac{\int_{B(z,\frac{\pi}2)}\langle z,u\rangle\,\mathcal{H}^{n-1}(du)}
{\mathcal{H}^{n-1}(B(z,\frac{\pi}2))}=\frac{2\kappa_{n-1}}{\mathcal{H}^{n-1}(\BS^{n-1})}.
\end{equation}
The following lemma shows how the left side increases when $B(z,\frac{\pi}2)$ is replaced by $B(z,\alpha)$ and
$0<\alpha\leq\frac{\pi}2-\varepsilon$.

\begin{lemma}
\label{monotonicitystab}
If $0<\alpha\leq\frac{\pi}2-\varepsilon$, $\varepsilon\in[0,\frac{\pi}6]$, $n\geq 2$ and $z\in \BS^{n-1}$, then
$$
\frac{\displaystyle{\int_{B(z,\alpha)}\langle z,u\rangle\,\mathcal{H}^{n-1}(du)}}{\mathcal{H}^{n-1}(B(z,\alpha))}
\ge (1+c_1\,\varepsilon)\, \frac{2\kappa_{n-1}}{\mathcal{H}^{n-1}(\BS^{n-1})},
$$
where $c_1>0$ depends on $n$.
\end{lemma}
\begin{proof} For $\alpha\in(0,\frac{\pi}2]$, let
$$
f(\alpha)=\frac{\int_{B(z,\alpha)}\langle z,u\rangle\,\mathcal{H}^{n-1}(du)}{\mathcal{H}^{n-1}(B(z,\alpha))}
=\frac{\int_{0}^{\alpha}(\cos s) \varrho(s)\,ds}
{\int_{0}^{\alpha}\varrho(s)\,ds},
$$
and hence
$$
f'(\alpha)=\frac{\varrho(\alpha)}
{\left(\int_{0}^{\alpha}\varrho(s)\,ds\right)^2}
\left(\cos\alpha\, \int_{0}^{\alpha}\varrho(s)\,ds-  \int_{0}^{\alpha}(\cos s) \varrho(s)\,ds\right)
<0
$$
as $\cos s>\cos\alpha$ for $0<s<\alpha$.

Since $f$ is monotone decreasing on $(0,\frac{\pi}2]$, it is sufficient to prove that
$f'(\alpha)\le -c_2$ for $\alpha\in [\frac{\pi}{3},\frac{\pi}{2}]$,
where $c_2>0$ is a constant  depending on $n$.
We observe that
$$
\int_{0}^{\alpha}(\cos s) \varrho (s)\,ds\geq
\cos\alpha\, \int_{\frac{\pi}6}^{\alpha}\varrho(s)\,ds +
\cos\frac{\pi}6\, \int_{0}^{\frac{\pi}6}\varrho(s)\,ds,
$$
and therefore
$$
\cos\alpha\, \int_{0}^{\alpha}\varrho(s) \,ds-  \int_{0}^{\alpha}(\cos s) \varrho(s)\,ds\leq
\left(\cos\alpha -\cos\frac{\pi}6\right)\int_{0}^{\frac{\pi}6}\varrho(s)\,ds<0.
$$
Since $\alpha\ge \frac{\pi}{3}>\frac{\pi}{6}$, we conclude that
$$
f'(\alpha)\leq \frac{(\sin \frac{\pi}3)^{n-2}}
{\left(\int_{0}^{\frac{\pi}2}\varrho(s)\,ds\right)^2}\,
\left(\cos\frac{\pi}3 -\cos\frac{\pi}6\right)\,\int_{0}^{\frac{\pi}6}\varrho(s)\,ds,
$$
which proves Lemma~\ref{monotonicitystab}.
\end{proof}

We slightly rephrase Theorem \ref{thmbkr0} from the introduction as follows.

\begin{theorem}\label{thmbkr}
Let $K\in\cK^n$.
\begin{enumerate}
\item[{\rm (i)}] Then $V_1(K)\ge 2 R(K)$,
with equality if and only if $K$ is a segment
\item[{\rm (ii)}] If  $V_1(K)\le (2+\varepsilon) R(K)$ for some   $\varepsilon\in[0,\varepsilon_0]$, then there exist  $z\in\R^n$ and a segment  $S$ of length
$2-\gamma_1\varepsilon^2$ such that $R(K) S\subset K-z\subset R(K)(S+\gamma_2\sqrt{\varepsilon}B^{n-1}_S)$, where $\varepsilon_0,\gamma_1,\gamma_2>0$ are constants depending on $n$.
\end{enumerate}
\end{theorem}

Note  that Theorem \ref{thmbkr0} remains true if $R(K)=0$ (as stated above). In this case $K$ is a point and all assertions hold  trivially. If $R(K)>0$, then the explicit use  of the vector $c$ can be avoided by considering a translate of the segment $S$.

\begin{proof}[Proof of Theorem \ref{thmbkr}] For the proof, we can assume that $R(K)>0$.
By homogeneity and translation invariance, we can then assume that  $B^n$ is the circumball of $K$, and hence $R(K)=1$.
It  follows that  the origin $o$ is contained in the convex hull of $\BS^{n-1}\cap K$.
Let $k$ be the minimal number of points of $\BS^{n-1}\cap K$ whose convex hull
contains $o$. Then $2\leq k\leq n+1$ by Carath\'eodory's theorem. Now let $x_1,\ldots,x_k\in \BS^{n-1}\cap K$ and $\lambda_1,\ldots,\lambda_k> 0$ be such that $\lambda_1+\cdots+\lambda_k=1$ and $\lambda_1x_1+\cdots+\lambda_kx_k=o$.
 For $i=1,\ldots,k$, we define the Dirichlet-Voronoi cell
$$
\Delta_i=\{x\in \BS^{n-1}:\langle x,x_i\rangle\geq \langle x,x_j\rangle \text{ for } j=1,\ldots,k\}.
$$
Then $\Delta_i$ is spherically convex,  starshaped with respect to $x_i$ and $\Delta_1,\ldots,\Delta_k$ partition
$\BS^{n-1}$, in particular, we have
 $\sum_{i=1}^k \mathcal{H}^{n-1}(\Delta_i)=\mathcal{H}^{n-1}(\BS^{n-1})$.
In fact, since $\lambda_1x_1+\cdots+\lambda_kx_k=o$ implies that for any $x\in \BS^{n-1}$, there is some $i\in\{1,\ldots,k\}$ such that
$\langle x,x_i\rangle \geq 0$, it follows that $\Delta_i\subset B\left(x_i,\frac{\pi}2\right)$ for $i=1,\ldots,k$.

For $x\in \Delta_i$, we have $h([x_1,\ldots,x_k],x)=\langle x,x_i\rangle$.
Hence, the monotonicity of intrinsic volumes, Lemma~\ref{monotonicity} and (\ref{intratio}) imply that
\begin{align}
V_1(K)&\geq V_1([x_1,\ldots,x_k])= \frac{n}{\kappa_{n-1}} V([x_1,\ldots,x_k],B^n[n-1])\label{bound01}\\
&=
\frac1{\kappa_{n-1}}\sum_{i=1}^k\int_{\Delta_i}\langle x,x_i\rangle\, \cH^{n-1}(dx)\nonumber\\
&\geq  \frac1{\kappa_{n-1}}\sum_{i=1}^k
\frac{2\kappa_{n-1}}{\mathcal{H}^{n-1}(\BS^{n-1})}\, \mathcal{H}^{n-1}(\Delta_i)=2.\label{bound02}
\end{align}
Now suppose that $V_1(K)=2$. Then equality holds in \eqref{bound01}, which yields that $K=[x_1,\ldots,x_k]$.
Moreover, Lemma \ref{monotonicitystab}
shows that strict inequality holds in \eqref{bound02} if $\Delta_i\neq B(x_i,\pi/2)$ for some $i\in\{1,\ldots,k\}$. Hence, $V_1(K)=2$ implies that $k=2$ and $K=[x_1,x_2]$.

\medskip

To prove the stability statement (ii), we may again assume that $R(K)=1$.

The main part of the proof consists in showing that if $\diam K\leq 2-\eta$ for some $\eta\in [0,1]$, then
\begin{equation}
\label{longsegment}
V_1(K)\ge 2+c_3\sqrt{\eta} \mbox{ \ for a constant $c_3>0$ depending on $n$.}
\end{equation}
For $\eta=0$, the assertion holds by (i). Henceforth, we assume that $\eta>0$, and therefore we also have $k>2$.
Since $\sum_{i=1}^k \mathcal{H}^{n-1}(\Delta_i)=n\kappa_n$, we may assume that
$\mathcal{H}^{n-1}(\Delta_1)\geq n\kappa_n/k$. For $i=2,\ldots,k$,
let $\digamma_i=\{x\in \Delta_1:\,\langle x,x_1\rangle=\langle x,x_i\rangle\}$, which is contained in the
hyper-sphere $\{x\in \BS^{n-1}:\langle x,x_1-x_i\rangle=0\}$ and compact. We may assume that $\digamma_i\neq \emptyset$ precisely for $i\in\{2,\ldots,\ell\}$, hence $2\le \ell\le k$ and $\dim\digamma_i=n-2$ for $i=2,\ldots,\ell$.  Let $\Gamma_i$ be  the union of all spherical geodesic arcs connecting $x_1$ to the points of $\digamma_i$. Then $\Gamma_i$ is compact and starshaped with respect $x_1$ and $\sum_{i=2}^\ell \mathcal{H}^{n-1}(\Gamma_i)=\mathcal{H}^{n-1}(\Delta_1)$.
Since $2\le \ell\le k\le n+1$, we may assume that
$$
\mathcal{H}^{n-1}(\Gamma_2)\geq \frac{n\kappa_n}{k(\ell-1)}\geq
\frac{\kappa_n}{n+1}.
$$
Note that $x_2\neq -x_1$ since $\diam K<2$. Let $u\in x_1^\bot\cap \BS^{n-1}\cap {\rm lin}\{x_1,x_2\}$
be the vector such that $\langle u,x_2\rangle>0$.
Writing $\beta=\angle(x_1,x_2)$ for the angle enclosed by $x_1$ and $x_2$, we have $x_2=x_1\cos\beta+u\sin\beta$.
We deduce from $\digamma_2\subset B(x_1,\frac{\pi}2)\cap (x_1-x_2)^\bot$ that
$\digamma_2\subset B(x_1,\frac{\pi}2)\cap B(u,\frac{\pi}2)$, and thus
$$
\Gamma_2\subset \mbox{$B(x_1,\frac{\pi}2)\cap B(u,\frac{\pi}2)\cap B(\frac{x_1-x_2}{\|x_1-x_2\|},\frac{\pi}2)$}.
$$
For $n\ge 3$ we define
$$
\Xi_1=\{x\in x_1^\bot\cap \BS^{n-1} : 0\leq \langle x,u\rangle\leq \tau\},
$$
where $\tau=\tau(n)\in (0,1]$, depending only on $n$,  is chosen such that
$$
\cH^{n-2}(\Xi_1)=\frac{\cH^{n-2}(\BS^{n-2})}{n(n+1)}.
$$
Then we connect each point of $\Xi_1$ to $x_1$ by a geodesic arc and take the union of all such arcs to get a compact subset $\Xi\subset B(x_1,\frac{\pi}{2})$ which is starshaped with respect to $x_1$ and satisfies
$$
\cH^{n-1}(\Xi)=\frac{\cH^{n-1}(\BS^{n-1})}{2n(n+1)}=\frac{\kappa_n}{2(n+1)}.
$$
For $n=2$, we define $\Xi$ as the empty set and set $\tau=1$.
\medskip

\noindent
{\bf Claim.} If $0\le \eta\le 1$, then
\begin{equation}
\label{xx1}
\langle x,x_1\rangle \geq 2^{-1}\tau\sqrt{\eta} \quad \text{for } x\in \Gamma_2\setminus\Xi.
\end{equation}
For the proof, we write $x$ in the form
$x=x_1\cos s +x_0\sin s$, where  $x_0\in x_1^\perp$ and $s\in [0,\frac{\pi}{2}]$.
Since $x\in \Gamma_2\setminus\Xi$, we
conclude further that $s\in (0,\frac{\pi}{2}]$  and  $\langle x_0,u\rangle\ge\tau>0$, if $n\ge 3$, and $x_0=u$
if $n=2$.

We first observe that
\begin{equation}\label{new0}
\frac{x_1-x_2}{\|x_1-x_2\|}=x_1\sin\frac{\beta}2-u\cos\frac{\beta}2
\end{equation}
and
$$
\sin\frac{\beta}2=\frac12\|x_1-x_2\|\leq 1-\frac{\eta}2,
$$
since $x_1,x_2\in K$ and $\diam K\leq 2-\eta$. In addition, we have
$$
\left(\tan \frac{\beta}2\right)^2\leq \frac{(1-\frac{\eta}2)^2}{1-(1-\frac{\eta}2)^2}
\leq \frac{4}{3\eta}.
$$
Since $x\in \Gamma_2$, we have $x\in B(\frac{x_1-x_2}{\|x_1-x_2\|},\frac{\pi}2)$, and hence
by \eqref{new0} it follows that
$$
\langle x,u\rangle\leq \langle x,x_1\rangle\tan \frac{\beta}2\le \langle x,x_1\rangle\frac{2}{\sqrt{3}}\frac{1}{\sqrt{\eta}}.
$$
If $s\in[\frac{\pi}{3},\frac{\pi}{2}]$, then
$$
\langle x,u\rangle=\langle x_0,u\rangle \sin s\ge\tau\sin s\ge \frac{\sqrt{3}}{2}\tau,
$$
and hence
$$
\langle x,x_1\rangle\ge \frac{3}{4}\tau\sqrt{\eta}\ge \frac{1}{2}\tau\sqrt{\eta}.
$$
If $s\in(0,\frac{\pi}{3}]$, then again
$$
\langle x,x_1\rangle =\cos s\ge\frac{1}{2}\ge \frac{1}{2}\tau\sqrt{\eta},
$$
since $0<\eta,\tau\le 1$. This proves the claim.
\bigskip

It follows from the construction of $\Gamma_2$ and $\Xi$ that
\begin{equation}\label{newereq}
\mathcal{H}^{n-1}(\text{cl}(\Gamma_2\setminus\Xi))=\mathcal{H}^{n-1}(\Gamma_2\setminus\Xi)\geq \frac{\kappa_n}{2(n+1)}=
\frac{1}{2n(n+1)}\, \mathcal{H}^{n-1}(\BS^{n-1}).
\end{equation}
We define $\alpha\in (0,\frac{\pi}{2})$ by $\cos\alpha =2^{-1}\tau\sqrt{\eta}\in (0,\frac{1}{2}]$, and
hence $\alpha\ge \frac{\pi}{3}$.
Then \eqref{xx1} implies that $\text{cl}(\Gamma_2\setminus\Xi)\subset B(x_1,\alpha)$ and
$$
2^{-1}\tau\sqrt{\eta}=\cos\alpha=\sin\left(\frac{\pi}{2}-\alpha\right)\le \frac{\pi}{2}-\alpha,
$$
so that $0<\alpha\le \frac{\pi}{2}-2^{-1}\tau\sqrt{\eta}$ and $2^{-1}\tau\sqrt{\eta}\le \frac{1}{2}\le \frac{\pi}{6}$.
Therefore, we can apply Lemma~\ref{monotonicity} to the topological closure of $\Gamma_2\setminus\Xi$, which is starshaped with respect to $x_1$, and also use Lemma~\ref{monotonicitystab} and  \eqref{newereq} to get
\begin{align*}
\int_{\Gamma_2\setminus \Xi}\langle x_1,x\rangle\,\mathcal{H}^{n-1}(dx)&\geq
\frac{\int_{B(x_1,\alpha)}\langle x_1,x\rangle\,\mathcal{H}^{n-1}(dx)}{\mathcal{H}^{n-1}(B(x_1,\alpha))}
\, \mathcal{H}^{n-1}(\Gamma_2\setminus \Xi)\\
&\ge (1+c_1 2^{-1}\tau\sqrt{\eta})\, \frac{2\kappa_{n-1}}{\mathcal{H}^{n-1}(\BS^{n-1})}
\, \mathcal{H}^{n-1}(\Gamma_2\setminus \Xi)\\
&\geq \frac{2\kappa_{n-1}}{\mathcal{H}^{n-1}(\BS^{n-1})}\, \mathcal{H}^{n-1}(\Gamma_2\setminus \Xi)
+\frac{\kappa_{n-1}c_12^{-1}\tau}{n(n+1)}\, \sqrt{\eta}.
\end{align*}
In addition, using again Lemma~\ref{monotonicity} we also have
\begin{align*}
\int_{\Gamma_2\cap \Xi}\langle x_1,x\rangle\,\mathcal{H}^{n-1}(dx)&\geq
\frac{2\kappa_{n-1}}{\mathcal{H}^{n-1}(\BS^{n-1})}\, \mathcal{H}^{n-1}(\Gamma_2\cap \Xi),\\
\int_{\Gamma_j}\langle x_1,x\rangle\,\mathcal{H}^{n-1}(dx)&\geq
 \frac{2\kappa_{n-1}}{\mathcal{H}^{n-1}(\BS^{n-1})}\, \mathcal{H}^{n-1}(\Gamma_j)\quad \mbox{ \ for $j\geq 3$},\\
\int_{\Delta_i}\langle x_i,x\rangle\,\mathcal{H}^{n-1}(dx)&\geq
 \frac{2\kappa_{n-1}}{\mathcal{H}^{n-1}(\BS^{n-1})}\, \mathcal{H}^{n-1}(\Delta_i)\quad \mbox{ \ for $i\geq 2$}.
\end{align*}
Summing up the individual contributions from the subsets, we get
\begin{align*}
&V_1(K)\geq V_1([x_1,\ldots,x_k])=
\frac1{\kappa_{n-1}}\sum_{i=1}^k\int_{\Delta_i}\langle x,x_i\rangle\, \cH^{n-1}(dx)\\
&\geq \frac{\kappa_{n-1}c_12^{-1}\tau}{n(n+1)}\, \sqrt{\eta}+\frac1{\kappa_{n-1}}\sum_{i=1}^k
\frac{2\kappa_{n-1}}{\mathcal{H}^{n-1}(\BS^{n-1})}\, \mathcal{H}^{n-1}(\Delta_i)
=  2+\frac{\kappa_{n-1}c_12^{-1}\tau}{n(n+1)}\, \sqrt{\eta},
\end{align*}
which completes the proof of (\ref{longsegment}).

\medskip

 Let us assume that $V_1(K)\leq 2+\varepsilon$ for some $\epsilon\in [0, \min\{c_3,1/2\}]$.
If $[y_1,y_2]\subset K$ is a longest segment in $K$, then  $\|y_1-y_2\|\geq 2-(\varepsilon/c_3)^2$
 by (\ref{longsegment}). For any $y\in K$, let $t$ be the distance of $y$ from $[y_1,y_2]$.
Since $2V_1([y_1,y_2,y])$ is the perimeter of $[y_1,y_2,y]$,
we get
\begin{align*}
2+\varepsilon\geq  V_1(K)&\geq V_1([y_1,y_2,y])\geq \frac{\|y_1-y_2\|}2+\sqrt{\frac{\|y_1-y_2\|^2}4+t^2}\\
&\geq 1-\frac{(\varepsilon/c_3)^2}2+\sqrt{\left(1-\frac{(\varepsilon/c_3)^2}{2}\right)^2+t^2}.
\end{align*}
Using that $( {\varepsilon}/({2c_3^{2}})+1)\varepsilon\leq 1$ and $(1+s)^2\leq 1+3s$ for
$s\in[0,1]$, we obtain that
$$
t^2\leq \left({5}/({2c_3^{2}})+3\right) \varepsilon,
$$
which in turn yields that $K\subset [y_1,y_2]+\sqrt{(3+3c_3^{-2})\varepsilon}\, (B^n\cap (y_2-y_1)^\perp)$.
\end{proof}

\section{Proofs of Theorem~\ref{mvolupper} and Theorem~\ref{mvolupperstab}}
\label{secmain}

Part (i) of Theorem \ref{thmbkr} is the main ingredient for the proof of Theorem~\ref{mvolupper}.

\begin{proof}[Proof of Theorem~\ref{mvolupper}]
We can assume that the circumball of $K$ has its centre at the origin. Then
\begin{align}
V(K,M[n-1])&=\frac{1}{n}\int_{\BS^{n-1}}h_K(u)\, S_{n-1}(M,du)\le \frac{1}{n}R(K)F(M)\label{astplus}\\
&\le\frac{1}{n}\frac{1}{2} V_1(K) 2V_{n-1}(M)=\frac{1}{n}V_1(K)V_{n-1}(M),\nonumber
\end{align}
where we used Theorem \ref{thmbkr} for the second inequality. If equality holds, then equality holds in
Theorem \ref{thmbkr}, since $V_{n-1}(M)>0$, and therefore $K=[-R e,Re]$ with $R=R(K)$ and for some $e\in\BS^{n-1}$. Moreover,
we then also have equality in the first inequality, which yields
$$
\int_{\BS^{n-1}}|\langle u,e\rangle|\, S_{n-1}(M,du)=S_{n-1}(M,\BS^{n-1}).
$$
This implies that the area measure of $M$ is concentrated in $\{-e,e\}$, hence $M$ is contained in
a hyperplane orthogonal to $e$.
\end{proof}

We now start to build the argument leading to  Theorem~\ref{mvolupperstab}, which is the stability version
of Theorem~\ref{mvolupper}.
Recall from the discussion in Section \ref{secpreliminaries}  that if $M\in\cK^n$ is at least $(n-1)$-dimensional  and $e\in \BS^{n-1}$, then
\begin{equation}\label{closer}
\int_{\BS^{n-1}}|\langle e,u\rangle|\, S_{n-1}(M,du)\le 2V_{n-1}(M),
\end{equation}
with equality if and only if $M\subset e^\perp$.

In the following  proposition,
we explore what can be said about $M$ if the integral on the left side of \eqref{closer} is $\varepsilon$-close to the upper bound.

\begin{proposition}
\label{surfacearea-stability}
Let $\varepsilon\in(0,\frac{1}{2}\left(\frac{1}{2n}\right)^n)$ and $e\in \BS^{n-1}$. Suppose that $M\in\cK^n$ is  at least $(n-1)$-dimensional and
\begin{equation}\label{astdplus}
\int_{\BS^{n-1}}|\langle e,u\rangle|\, S_{n-1}(M,du)\geq (1-\varepsilon) 2V_{n-1}(M).
\end{equation}
Then there is a unit vector $v\in \BS^{n-1}$ such that $ w_M(v)\leq c_4 r\sqrt{\varepsilon}$ and
$\langle e,v\rangle\geq 1-c_5\varepsilon$,
where   $c_4\le 48 n^2\sqrt{6}^{\,n}$,
$c_5\le (10 n)^4(2n)^n$ are positive constants and
 $r$ is the maximal radius of an $(n-1)$-ball in $M\vert e^\bot$.
\end{proposition}

\noindent{\bf Remark.} The lemma is essentially optimal for $n\ge 3$, in the sense that one cannot conclude in general that
$w_M(e)\leq c_0\, V_{n-1}(M)^{\frac1{n-1}} r \sqrt{\varepsilon}$ with a constant $c_0>0$. To show this, let
$v\in \BS^{n-1}$ with $\langle e,v\rangle=1-\frac{\varepsilon}2$, and let $e_1,\ldots,e_n$ be an orthonormal basis such that $e_1=v$ and
$e\in{\rm lin}\{e_1,e_2\}$. For large $\lambda$, we define
$$
M=\left[\pm\sqrt{\varepsilon} e_1,\pm \lambda e_2,\pm n e_3,\ldots,\pm n e_n\right],
$$
which satisfies
$1<r<n$, $V_{n-1}(M)<c_6\lambda$, $\mathcal{H}^{n-1}(M|e^\bot)\geq (1-\varepsilon) V_{n-1}(M)$ and
$w_M(e)>c_7\lambda\sqrt{\varepsilon}> c_8\lambda^{\frac{n-2}{n-1}} V_{n-1}(M)^{\frac1{n-1}} r \sqrt{\varepsilon}$
with constants $c_6,c_7,c_8>0$ depending only on $n$, if $\lambda>0$ is large and $\varepsilon>0$ is small enough.

\begin{proof}[Proof of Proposition \ref{surfacearea-stability}]
For the proof, we can assume that $M$ is $n$-dimensional. This follows from an approximation argument (which will require adjustments of $M$, $\varepsilon$ and $r$).

Let us consider some $n$-dimensional convex body $L\in\cK^n$ with
$V_{n-1}(L)\leq V_{n-1}(M)$ and
$\mathcal{H}^{n-1}(L\vert e^\bot)=\mathcal{H}^{n-1}(M\vert e^\bot)$, and let
\begin{align*}
\partial_+L=\{y\in\partial' L:\langle e,\nu_L(y)\rangle> 0\},\quad
\partial_-L=\{y\in\partial' L:\langle e,\nu_L(y)\rangle< 0\}.
\end{align*}
We first prove an auxiliary claim which will be applied twice in the subsequent argument.

\bigskip

\noindent {\bf Claim.} Suppose there exist constants $\eta,\gamma>0$ and a compact convex set $C\subset L\vert e^\bot$
with $\mathcal{H}^{n-1}(C)=\gamma\, \mathcal{H}^{n-1}(L\vert e^\bot)$ such that  any
$y\in\partial_+ L$ with $y\vert e^\bot\in C$ satisfies $\tan\angle(e,\nu_L(y))\geq \eta$. Then
\begin{equation}
\label{yCnotM}
\eta\leq \frac{4\sqrt{\varepsilon}}{\sqrt{\gamma}}\,\mbox{ \ provided\, }\varepsilon<\frac{\gamma}{4}.
\end{equation}

To prove the claim,  let $Y$ denote the set of all $y\in\partial_+ L$ with $y\vert e^\bot\in C$.
For any $y\in Y$, we have
$$
0<\langle  e,\nu_L(y) \rangle=\cos \angle(e,\nu_L(y))=\sqrt{\frac1{1+(\tan\angle(e,\nu_L(y)))^2}}\leq \sqrt{\frac1{1+\eta^2}}.
$$
It follows that
\begin{equation}\label{eqna}
\gamma\,\mathcal{H}^{n-1}(L|e^\bot)=\mathcal{H}^{n-1}(C)=
\int_{Y}\vert \langle e,\nu_{L}(y)\rangle|\, \mathcal{H}^{n-1}(dy)\leq
\sqrt{\frac1{1+\eta^2}}\, \mathcal{H}^{n-1}(Y).
\end{equation}
Furthermore, we have
\begin{align}
\left(2-\gamma\right)\mathcal{H}^{n-1}(L\vert e^\bot)&=
\mathcal{H}^{n-1}(L\vert e^\bot)+\mathcal{H}^{n-1}
\left((L\vert e^\bot)\setminus C\right)\nonumber\\
&=
\int_{(\partial' L)\setminus Y}|\langle e,\nu_{L}(y)\rangle|\, \mathcal{H}^{n-1}(dy)\nonumber\\
&\leq \mathcal{H}^{n-1}\left((\partial L)\setminus Y\right).\label{eqnb}
\end{align}
From \eqref{astdplus}, \eqref{eqna} and \eqref{eqnb}, we deduce that
\begin{align*}
(1+2\varepsilon)\, 2\,\mathcal{H}^{n-1}(L\vert e^\bot)&=(1+2\varepsilon)\, 2\,\mathcal{H}^{n-1}(M\vert e^\bot)\\
&\ge (1+2\varepsilon)(1-\varepsilon)\, \mathcal{H}^{n-1}(\partial M)\\
& \ge \mathcal{H}^{n-1}(\partial M)\geq \mathcal{H}^{n-1}(\partial L)\\
&=\mathcal{H}^{n-1}\left((\partial L)\setminus Y\right)+
\mathcal{H}^{n-1}\left(Y\right)\\
&\geq\left(2-\gamma+
\gamma\sqrt{1+\eta^2}\right)\mathcal{H}^{n-1}(L\vert e^\bot),
\end{align*}
and hence $4\varepsilon\geq \gamma(\sqrt{1+\eta^2}-1)$. Since $(1+s)^2\leq 1+4s$ for
$s\in (0,1)$,  we conclude (\ref{yCnotM}), which proves the claim.

\bigskip

We set
$$
t=\max\left\{\mathcal{H}^1((x+\R e)\cap M):\,x\in M|e^\bot \right\}
$$
and write $r$ to denote the maximal radius of $(n-1)$-balls in $M\vert e^\bot$.
Possibly after a translation of $M$, there exists an origin symmetric ellipsoid $E$  such that
\begin{equation}
\label{John}
E\subset M\subset n E
\end{equation}
according to John's theorem. By changing the orientation of $e\in\BS^{n-1}$, if necessary, there is a positive $\tau$ such that $\tau e\in \partial E$
and $\tau\le  \frac{t}{2}\le n\tau$, and hence $\tau\in [\frac{t}{2n},\frac{t}2]$. Let $v\in \mathbb{S}^{n-1}$ be the exterior unit normal at $\tau e$ to $E$.
It follows that
$$
w_M(v)\leq w_{nE}(v)=2nh_E(v)
=2n\langle \tau e,v\rangle=2n\tau\langle e,v\rangle\le 2n\tau,
$$
and thus
\begin{equation}
\label{efE}
w_M(v)\leq nt.
\end{equation}

We prove Proposition~\ref{surfacearea-stability} in two steps. First, we bound $t$ from above, then we establish a lower bound for $|\langle e,v\rangle|$.

\bigskip

\noindent {\bf Step 1.} We show that $t\leq c_9r\sqrt{\varepsilon}$ with a constant
$c_9\le 48n\sqrt{6}^{n-1}$.

Let $w\in \BS^{n-1}\cap e^\bot$ be such that $h_E(w)=\min\{h_E(u):u\in \BS^{n-1}\cap e^\bot\}$, which equals
the inradius of $E\vert e^\perp$, and hence
$h_E(w)\leq r$. In turn, for $y\in M$ we deduce from (\ref{John}) that
$$
|\langle y,w\rangle|\le \max\{h_M(w),h_M(-w)\}\le h_{nE}(w)=nh_E(w)\le nr,
$$
that is,
\begin{equation}
\label{Mwidthr}
|\langle y,w\rangle|\leq nr \mbox{ \ for \ }y\in M.
\end{equation}

We may choose an orthonormal basis $e_1,\ldots,e_n$ of $\R^n$ such that $e_1=e$ and $e_2=w$, and
let $M'$ be the convex body resulting from $M$ via successive Steiner symmetrizations with respect to
 $e_1^\bot,\ldots,e_n^\bot$. Then $M'$ satisfies
$$
|\langle y,w\rangle|\leq nr \mbox{ \ for \ }y\in M',\quad \pm \frac12\,te\in\partial M', \quad
 (rB^n\cap e^\bot)\subset M',
$$
$M'$ is symmetric with respect to the coordinate subspaces  $e_1^\perp,\ldots,e_n^\perp$, and
hence in particular $M'$ is centrally symmetric,
$\mathcal{H}^{n-1}(M'|e^\bot)=\mathcal{H}^{n-1}(M|e^\bot)$,
$M'|e^\bot=M'\cap e^\bot$ and
$V_{n-1}(M')\leq V_{n-1}(M)$. Finally, we consider the double cone
$$
\widetilde{M}={\rm conv}\left\{ \frac12\,te,- \frac12\,te,M'\cap e^\bot\right\},
$$
which satisfies
$$
|\langle y,w\rangle|\leq nr \quad \text{for }y\in \widetilde{M},\quad \pm \frac12\,te\in\partial \widetilde{M}
,\quad (rB^n\cap e^\bot)\subset \widetilde{M},
$$
$\widetilde{M}$ is symmetric with respect to the coordinate subspaces  $e_1^\perp,\ldots,e_n^\perp$ (and hence centrally symmetric),
$\mathcal{H}^{n-1}(\widetilde{M}|e^\bot)=\mathcal{H}^{n-1}(M|e^\bot)$ and
$V_{n-1}(\widetilde{M})\leq V_{n-1}(M)$.

For $\varrho=h_{\widetilde{M}}(w)$, we have $\varrho w\in \partial \widetilde{M}$, and hence
$\varrho\leq nr$. To prepare an application of (\ref{yCnotM}) with $L=\widetilde{M}$, we consider
$$
C=\frac56\,\varrho w+\frac16\,\left(\widetilde{M}\vert e^\bot\right)\subset \widetilde{M}\vert e^\bot
\quad\text{ with }\gamma=6^{-(n-1)}.
$$
Let $U=e^\bot\cap w^\bot$. Let $Y$ denote the set of all $y\in\partial_+\widetilde{M}$ such that $y|e^\perp\in C$. For $y\in Y$, we
set $\alpha=\angle(e,\nu_{\widetilde{M}}(y))$, thus
$y=se+pw+v_1$ and  $\nu_{\widetilde{M}}(y)=e\cos \alpha+qw+v_2$, where
$$
v_1\in \frac16\,\left(\widetilde{M}\cap U\right),\quad v_2\in U
,\quad s\geq 0,\quad \frac23\,\varrho\leq p\leq \varrho
,\quad q\leq \sin\alpha.
$$
Since $\widetilde{M}$ is a double cone, there exists $z\in e^\bot\cap \partial  \widetilde{M}$ such that
$y\in[z,\frac12\,te]$ and $y\vert e^\bot\in[z,o]$. We deduce that
$$
\frac{s}{\frac12\,t}=\frac{\|z-(y\vert e^\bot)\|}{\|z\|}=
\frac{\langle z-(y\vert e^\bot),w\rangle}{\langle z,w\rangle}\leq \frac13,
$$
thus $s\leq \frac16\,t$.
As $\frac12\,te\in \widetilde{M}$, we have
$$
0\leq\left\langle y-\mbox{$\frac12\,te$},\nu_{\widetilde{M}}(y)\right\rangle=\mbox{$(s-\frac12\,t)\cos\alpha$} +pq+\langle v_1,v_2\rangle,
$$
which yields
$$
-\langle v_1,v_2\rangle\leq \mbox{$(s-\frac12\,t)\cos\alpha$}+pq.
$$
Therefore, since $2v_1\in \widetilde{M}$ we obtain
\begin{align*}
0&\leq \langle y-2v_1,\nu_{\widetilde{M}}(y)\rangle=
\mbox{$s\cdot\cos\alpha$} +pq-\langle v_1,v_2\rangle\leq
\mbox{$(2s-\frac12\,t)\cos\alpha$} +2pq\\
&\leq
-\mbox{$\frac16\,t\cos\alpha$} +2\varrho\sin\alpha\leq -\mbox{$\frac16\,t\cos\alpha$} +2nr\sin\alpha,
\end{align*}
which implies that $\tan\alpha\geq \frac{t}{12nr}$. Now an application of (\ref{yCnotM}) proves the estimate of Step 1.

\bigskip

\noindent{\bf Step 2.} We show that $\langle e,v\rangle \geq 1-c_8\varepsilon$.

Let $\beta=\angle (e,v)\in[0,\frac{\pi}2)$, and let $\widetilde{w}\in \BS^{n-1}\cap e^\bot$ be such that $v=e\cos\beta+\widetilde{w}\sin\beta$. Since the shadow boundary of $E$ in direction $e$ lies in a hyperplane and by \cite[Theorem 1]{MR89}, it follows from the definition of $v$ that $E|e^\bot=(E\cap v^\bot)|e^\bot$. Hence
(\ref{John}) yields that
\begin{equation}
\label{efJohn}
(E\cap v^\bot)|e^\bot\subset \pm M|e^\bot\subset n(E\cap v^\bot)|e^\bot.
\end{equation}
Since there is some $a\in e^\perp$ such that $(rB^n\cap e^\perp) +a\subset M|e^\bot$ it follows that
$$\frac{r}{n}(B^n\cap e^\perp)\pm \frac{a}{n}\subset \pm \frac{1}{n}(M\vert e^\perp)\subset E\vert e^\perp,$$
which shows that
$$
\frac{r}n\,(B^n\cap e^\bot)\subset (E\cap v^\bot)|e^\bot,
$$
and we deduce that
\begin{equation}
\label{rholower}
\varrho=\max\left\{\lambda>0:\lambda\,\widetilde{w}\in (E\cap v^\bot)|e^\bot\right\}\geq \frac{r}n.
\end{equation}

In order to apply (\ref{yCnotM}), we now choose $L=M$ and
$$
C=\frac12\,(E|e^\bot)=\frac12(E\cap v^\bot)|e^\bot\subset M\vert e^\bot,
$$
which satisfies
$$
\mathcal{H}^{n-1}(C)=\widetilde{\gamma}\,\mathcal{H}^{n-1}(M|e^\bot),\quad \widetilde{\gamma}\ge (2n)^{-(n-1)},
$$
according to (\ref{efJohn}).

As before, we define $Y$ as the set of all $y\in\partial_+ M$ with $y\vert e^\bot\in X$.
Then, for $y\in Y$ we have
$$
x=y|e^\bot\in  \frac12(E\cap v^\bot)|e^\bot.
$$
It follows that
$$
x-\frac{\varrho}2\,\widetilde{w}\in (E\cap v^\bot)|e^\bot,
$$
and hence $x-\frac{\varrho}2 \widetilde{w}=z|e^\bot$ for a $z\in E\cap v^\bot\subset M$.
In addition, the definition of $t$ implies the existence of an $s\in[0,t]$ such that
$y-se\in E\cap v^\bot$.
Since $v=e\cos\beta+\widetilde{w}\sin\beta$, we deduce that
$$
z-(y-se)=-\frac12\,\varrho\,\widetilde{w}+ \frac12\,\varrho(\tan\beta)\, e,
$$
thus
$$
y-z=\frac12\,\varrho\,\widetilde{w}+ \left(s-\frac12\,\varrho\tan\beta\right)\,e.
$$
We set  $\alpha=\angle(e,\nu_M(y))\in [0,\frac{\pi}2)$, and hence
$$
\nu_M(y)=(\cos\alpha)e+\widetilde{p}\,\widetilde{w}+\widetilde{v},
$$
where  $\widetilde{v}\in \left({\rm lin}\{\widetilde{w},e\}\right)^\bot$ and
$\widetilde{p}^{\,2}+\|\widetilde{v}\|^2=(\sin\alpha)^2$.

We deduce from Step~1, $0\leq s\leq t$ and $\varrho\geq r/n$ (compare (\ref{rholower})) that
\begin{align*}
0&\leq \langle y-z,\nu_M(y)\rangle=
\left\langle \frac12\,\varrho\,\widetilde{w}+  \left(s-\frac12\,\varrho\tan\beta\right)e,(\cos\alpha)e+\widetilde{p}\,\widetilde{w}+\widetilde{v}\right\rangle\\
&=\left(s-\frac12\varrho\tan\beta\right)\cos\alpha +\frac{\widetilde{p}}2\,\varrho\\
&\leq
\left(c_9n\varrho\sqrt{\varepsilon}-\frac12\varrho\tan\beta\right)\cos\alpha +\frac{1}{2}\sin\alpha\,\varrho,
\end{align*}
thus $\tan\alpha\geq \tan\beta-2c_9 n\sqrt{\varepsilon}$. If $\tan\beta-2c_9 n\sqrt{\varepsilon}>0$, we conclude from (\ref{yCnotM})
that
$$
\tan\beta-2c_9 n\sqrt{\varepsilon}\leq 4(2n)^{\frac{n-1}2}\sqrt{\varepsilon},
$$
which in turn yields $\beta\leq\tan\beta\leq c_{10}\sqrt{\varepsilon}$
with $c_{10}\le 96 n^2\sqrt{6}^{n-1}+4\sqrt{2n}^{n-1}\le (10 n)^2\sqrt{2n}^n$.
If $\tan\beta\le 2c_9 n\sqrt{\varepsilon}$, we directly arrive at the same conclusion.
Therefore, in any case $\langle e,v\rangle=\cos\beta\geq 1-\frac12\, c_{10}^2\varepsilon$, completing Step~2.
\end{proof}

Finally, we combine the stability estimates above to derive  Theorem~\ref{mvolupperstab} in the following form.

\begin{theorem}\label{stabreverseMI}
Let $K,M\in\cK^n$ with $\dim(K)\ge 1$ and $\dim(M)\ge n-1$. Suppose that
\begin{equation}\label{eqbneubb}
V(K,M[n-1])\ge (1-\varepsilon) \frac{1}{n}V_1(K)V_{n-1}(M)
\end{equation}
for some  $\varepsilon\in [0,\varepsilon_0]$. Then there are a segment $S$ of length $(2-c_{11}\varepsilon^2)$ parallel to $e\in \BS^{n-1}$
such that $R(K)S\subset K\subset R(K)(S+c_{12}\sqrt{\varepsilon}B^{n-1}_S)$, and there is a vector $v\in\BS^{n-1}$ such that $\langle e,v\rangle \ge 1-c_{13}\sqrt{\varepsilon}$
and $w_M(v)\le c_{14}\,r\,\varepsilon^{\frac{1}{4}}$, where $r$ is the maximal radius of $(n-1)$-balls in $M\vert e^\perp$ and $\varepsilon_0,c_{11},\ldots,c_{14}>0$ are constants
depending on $n$.
\end{theorem}

\begin{proof}
Suppose that \eqref{eqbneubb} is satisfied for some $\varepsilon\in [0,\varepsilon_0]$, where $\varepsilon_0\in (0,\frac{1}{2})$ is sufficiently small.
A combination of \eqref{astplus} and \eqref{eqbneubb} yields
$$
\frac{1}{n} R(K) F(M)\ge V(K,M[n-1])\ge (1-\varepsilon)\frac{1}{n} V_1(K) V_{n-1}(M),
$$
and hence
$$
V_1(K)\le \frac{1}{1-\varepsilon} 2R(K)\le (1+2\varepsilon)2 R(K)=(2+4\varepsilon)R(K).
$$
By Theorem \ref{thmbkr}, there is a segment $S$ of length $2-c_{15} \varepsilon^2$ such that $R(K)S\subset K\subset R(K)(S+c_{16}\sqrt{\varepsilon} B^{n-1}_S)$ (a translation vector can be absorbed into $S$).
In particular,
$$
V_1(K)\ge V_1(S) R(K)\ge (2-c_{15}\varepsilon^2)R(K).
$$
If $S$ has direction $e\in \mathbb{S}^{n-1}$, then again by the assumption we obtain
\begin{align*}
&(1-\varepsilon)(2-c_{15}\varepsilon^2)\frac{1}{n}R(K)V_{n-1}(M)\\
&\le (1-\varepsilon)\frac{1}{n}V_1(K)V_{n-1}(M)\\
&\le V(K,M[n-1])=\frac{1}{n}\int_{\mathbb{S}^{n-1}}h_K(u)\, S_{n-1}(M,du)\\
&\le \frac{1}{n}\int_{\mathbb{S}^{n-1}}\left(h_S(u)+c_{16}\sqrt{\varepsilon}\right)R(K)\, S_{n-1}(M,du)\\
&=\frac{R(K)}{n}\frac{1}{2}V_1(S)\int_{\mathbb{S}^{n-1}}|\langle u,e\rangle|\, S_{n-1}(M,du)+c_{16}\sqrt{\varepsilon}\frac{R(K)}{n}2V_{n-1}(M),
\end{align*}
and hence
\begin{align*}
\int_{\mathbb{S}^{n-1}}|\langle u,e\rangle|\, S_{n-1}(M,du)&\ge (1-\varepsilon)2V_{n-1}(M)-c_{16}\frac{2\sqrt{\varepsilon}}{2-c_{15}\varepsilon^2}2V_{n-1}(M)\\
&\ge 2V_{n-1}(M)(1-c_{17}\sqrt{\varepsilon}).
\end{align*}
An application of Proposition \ref{surfacearea-stability} completes the proof.
\end{proof}


\end{document}